\documentclass[12pt,leqno]{article}

\usepackage{latexsym,amsmath,amssymb,amsthm,amsfonts,bm}
\usepackage[active]{srcltx}     
\usepackage[all]{xy}
\usepackage{verbatim}

\renewcommand{\epsilon}{\varepsilon}

\newcommand{\R}{\mathbb{R}}
\newcommand{\Q}{\mathbb{Q}}
\newcommand{\N}{\mathbb{N}}
\newcommand{\Z}{\mathbb{Z}}
\newcommand{\C}{\mathbb{C}}
\newcommand{\h}{\mathfrak{H}}
\newcommand{\tr}{\mathrm{tr}}
\newcommand{\res}{\mathrm{res}}

\newcommand{\e}{\bm{\mathrm{e}}}

\newcommand{\slz}{\mathrm{SL}_2(\Z)}
\newcommand{\mpz}{\mathrm{Mp}_2(\Z)}

\newtheorem{theorem}{Theorem}
\newtheorem{lemma}{Lemma}
\newtheorem{proposition}{Proposition}
\newtheorem{remark}{Remark}

\author{Maryna S. Viazovska}
\begin{document}
\title{CM Values of Higher Green's Functions}
\maketitle 
\maketitle
\section{Introduction}
For any integer $k > 1$
and subgroup $\Gamma\subset \mathrm{PSL}_2(\Z)$ of finite index there is a unique function $G_k^{\Gamma\backslash \h}$ on the
product of two upper half planes $\h\times\h$  which satisfies the following conditions:\begin{description}
\item{(i)} $G_k^{\Gamma\backslash \h}$ is a smooth function on $\h\times\h\setminus\{(\tau,\gamma\tau),\tau\in\h,\gamma\in\Gamma\}$ with values in
$\R$.
\item{(ii)} $G_k^{\Gamma\backslash \h}(\tau_1,\tau_2)=G_k^{\Gamma\backslash \h}(\gamma_1\tau_1,\gamma_2\tau_2)$ for all $\gamma_1,\gamma_2\in\Gamma$.
\item{(iii)} $\Delta_i G_k^{\Gamma\backslash \h}=k(1-k)G_k^{\Gamma\backslash \h}$ , where $\Delta_i$ is the hyperbolic Laplacian with respect to
the $i$-th variable, $i = 1, 2$.
\item{(iv)} $G_k^{\Gamma\backslash \h}(\tau_1,\tau_2) = m\log |\tau_1-\tau_2| + O(1)$ when $\tau_1$ tends to $\tau_2$ ($m$ is the order of
the stabilizer of $\tau_2$, which is almost always 1).
\item{(v)} $ G_k^{\Gamma\backslash \h}(\tau_1,\tau_2)$ tends to $0$ when $\tau_1$ tends to a cusp.
\end{description}
This function is called the Green's function.

In this note we will concentrate on the case $\Gamma=\slz$ and we will write simply $G_k$ for $G_k^{\slz\backslash \h}$.



Let $f$ be a modular function. Then the action of Hecke operator $T_m$ on $f$ is given by
$$f(\tau)|T_m =m^{-1}\sum_{\left( \vcenter{\xymatrix@R=0pt@C=0pt@M=1pt{\scriptstyle a&\scriptstyle b\\ \scriptstyle c& \scriptstyle d}}\right)\in\, \slz \backslash \mathcal{M}_m}(c\tau+b)^k f\left(\frac{a\tau+b}{c\tau+d}\right), $$  where $\mathcal{M}_m$ denotes the set of $2\times 2$ integral matrices of determinant $m$.

Green's functions $G_k$ have the property
$$G_k(\tau_1,\tau_2)|T_m^{\,\tau_1}=G_k(\tau_1,\tau_2)|T_m^{\,\tau_2},$$
where $T_m^{\,\tau_i}$ denotes the Hecke operator with respect to variable $\tau_i$, $i=1,2$.

Denote by $S_{2k}(\slz)$ the space of cusp forms of weight $2k$ on full modular group. 
\begin{proposition}\label{g_lambda}Let $k>1$ and $ \bm{\lambda}=\{\lambda_m\}_{m=1}^\infty\in\oplus_{m=1}^\infty \Z$. Then the following are equivalent\begin{description}\item{(i)} $\sum_{m=1}^\infty \lambda_m a_m=0$ for any cusp form $$f=\sum_{m=1}^\infty a_m q^m\in S_{2k}(\slz)$$ \item{(ii)} There exists a weakly holomorphic modular form $$g_{\bm{\lambda}}(\tau)=\sum_{m=1}^\infty\lambda_mq^{-m}+O(1)\in M^!_{2-2k}(\slz).$$\end{description}  \end{proposition}
The proof of this proposition can be found, for example, in \cite{Bo2} Section 3. The space of obstructions to finding modular forms of weight $2-2k$ with
given singularity at the cusp and the space of holomorphic modular forms of weight $2k$ can be both identified with cohomology groups of line bundles over modular curve. The statement follows from Serre duality between these spaces.

We call a $\bm{\lambda}$ with these properties a relation for $S_{2k}(\slz)$. Note that the function  $g_{\bm{\lambda}}$ in (ii) is unique and has integral Fourier coefficients.

For a relation $\bm{\lambda}$ denote
$$G_{k,\bm{\lambda}}:=\sum_{m=1}^\infty \lambda_m m^{k-1} G_k(\tau_1,\tau_2)|T_m . $$
The following conjecture was formulated in \cite{GZ} and \cite{GKZ}.\\

\noindent \textbf{Conjecture.} Suppose $\bm{\lambda}$ is a relation for $S_{2k}(\slz)$. Then for any
two CM points $\mathfrak{z}_1$, $\mathfrak{z}_2$ of discriminants $D_1$, $D_2$ there is an algebraic number $\alpha$ such
that $$G_{k,\bm{\lambda}}(\mathfrak{z}_1,\mathfrak{z}_2)=(D_1 D_2)^{\frac{1-k}{2}}\log\alpha.$$


In many particular cases this conjecture was proven by A. Mellit in his Ph. D. thesis \cite{M Anton}.

In this note we present a proof of the conjecture in the case when $\mathfrak{z}_1,\mathfrak{z}_2$ lie in the same imaginary quadratic field $\Q(\sqrt{-D})$.

Two main ingredients of the proof are the theory of Borcherds lift developed in \cite{Bo1} and a notion of see-saw identities introduced in \cite{Ku84}.  Firstly, following ideas given in \cite{Br} we prove that the Green's function can be realized as a Borcherds lift of an eigenfunction of Laplace operator. This allows as to extend a method given in \cite{Sch}, that is  to analyze CM values of Green's function using  see-saw identities. Applying see-saw identities we prove that a CM-value of higher Green's function is equal to a CM-value of certain meromorphic modular function with algebraic Fourier coefficients.

\section{Differential operators}
For $k\in\Z$ define differential operators
$$R_k =\frac{1}{2\pi i}(\frac{\partial }{\partial \tau}+\frac{k}{\tau-\bar{\tau}}),\;\; L_k =\frac{1}{2\pi i}(\tau-\bar{\tau})^2\frac{\partial }{\partial \bar{\tau}}, $$
$$\Delta_k =-4\pi^2\,R_{k-2}L_k= -4\pi^2\,(L_{k+2}R_k-k)=(\tau-\bar{\tau})^2\frac{\partial^2 }{\partial \tau \partial \bar{\tau}}+k(\tau-\bar{\tau})\frac{\partial }{\partial \bar{\tau}}. $$


For integers $k,w$ we
denote by $F_{k,w}$ the space of functions of weight $w$ satisfying
$$\Delta f = (k(1 - k) + \frac{w(w-2)}{4} )f.$$

\begin{proposition}\label{diff} The spaces $F_{k,w}$ satisfy the following properties:\begin{description}\item{(i)} The space $F_{k,w}$ is invariant under the action of the group $\mathrm{SL}_2(\R)$,
\item{(ii)} The operator $R_w$ maps $F_{k,w}$ to $F_{k,w+2}$,
\item{(iii)} The operator $L_w$ maps $F_{k,w}$ to $F_{k,w-2}$.
\end{description}
\end{proposition}
For a modular form of weight $k$ we will use the notation $$R^{r}f=R_{k+2r-2}\circ\cdots\circ R_k f. $$
Denote $f^{(s)}:=\frac{1}{(2\pi i)^s}\frac{\partial^s}{\partial \tau^s} f$. We have (see equation (56) in \cite{1-2-3})
\begin{equation}\label{56} R_l^k(f)=\sum_{r=0}^k(-1)^{k-r}{k\choose r}\frac{(l+r)_{k-r}}{(4\pi y)^{k-r}}f^{(r)},\end{equation}where $(a)_m = a(a+1)\cdots (a+m-1)$ is the Pochhammer symbol.
For a modular forms $f$ and $g$ of weight $k$ and $l$ the Rankin-Cohen bracket is defined by
$$[f,g]=lf'g-kfg', $$
and more generally
$$[f,g]_r=\sum_{s=0}^r (-1)^s {k+r-1\choose s} {l+r-1\choose r-s}f^{(r-s)}g^{(s)}. $$
The function $[f,g]_r$ is a modular form of the weight $k+l+2r$.
Note that $${k\choose r}:=\frac{(k)_r}{s!}$$ is defined for $s\in \N $ and $k\in\R$.

We will need the following proposition.
\begin{proposition} \label{RC1} Suppose that $f$ and $g$ are modular forms of weight $k$ and $l$ respectively. Then, for integer $r\geq 0$ we have $$R^{r}(f)\, g =a[f,g]_{r}+R(\sum_{s=0}^{r-1} b_{s} R^{s}(f)\,R^{r-s-1}(g))$$
where
$$a= {k+l+2r-2\choose r}^{-1}$$
and $b_{s}$ are some rational numbers. \end{proposition}\begin{proof} The operator $R$ satisfies the following property $$R(fg)=R(f)g+fR(g).$$ Thus, the sum $$\sum_{i+j=r}a_i R^i(f)R^j(g) $$ can be written as $$R(\sum_{i+j=r-1}b_i R^i(f)R^j(g)) $$ for some numbers $b_i$ if and only if $\sum_{i=0}^r(-1)^ia_i=0$. For the Rankin-Cohen brackets the following identity holds \begin{equation}\label{rc1} [f,g]_r=\sum_{s=0}^r  (-1)^s{k+r-1\choose s} {l+r-1\choose r-s}R^{(r-s)}(f)R^s(g). \end{equation} We will use the following standard identity $$\sum_{s=0}^r {k+r-1\choose s} {l+r-1\choose r-s}= {k+l+2r-2\choose r}.$$ It follows from the above formula and \eqref{rc1} that the sum $${k+l+2r-2\choose r}R^r(f)g-[f,g]_r$$ can be written in the form $$R(\sum_{i+j=r-1}b_i R^i(f)R^j(g)). $$ This finishes the proof.\end{proof}

\begin{proposition}\label{RC2} Suppose that $f$ is a real analytic modular form of weight $k-2$ and  $g$ is a holomorphic modular form of weight $k$. Then, for a compact region $F\subset \h$ we have  $$\int_F R_{k-2}(f)\,\bar g\, y^{k-2}\, dx\, dy=\int_{\partial F} f\,\bar{g}\,y^{k-2}\,(dx-idy) .$$ \end{proposition}\begin{proof} Follows from Stokes' theorem. \end{proof}
Denote by $K_\nu$ the Bessel K-function $$I_\nu(x)=\sum_{n=0}^\infty \frac{(x/2)^{\nu+2n}}{n!\Gamma(\nu+n+1)},\;\;K_\nu(x)=\frac{\pi}{2}\,\frac{I_{-\nu}(x)-I_\nu(x)}{\sin(\pi\nu)}.$$The function $K_\nu$ becomes elementary for $\nu\in
\Z+\frac12$. We have
$$K_{k+\frac 12}(x)=\frac{(\pi/2)^{\,\frac12}}{x^{k+\frac 12}}\,e^{-x}\,h_k(x), $$
for $k\in\Z_{\geq 0}$, where $h_k$ is the polynomial $$h_k(x)=\sum_{r=0}^k\frac{(k+r)!}{2^r r! (k-r)!}x^{k-r}.$$

The following statement follows immediately from equation \eqref{56}.
\begin{proposition}\label{Bessel}For $k\in\Z_{>0}$ the following identity holds
$$R_{-2k}^k(\e(n\tau))=2\, y^{\frac12}\,n^{k+\frac 12}\, K_{k+1/2}(2\pi n y)\,\e(nx). $$

\end{proposition}

\section{Vector valued modular forms}
Recall that the group $\mathrm{SL}_2(\Z)$ has a double cover $\mathrm{Mp}_2(\Z)$ called the metaplectic group whose elements can be written in the form
$$\left( \left( \vcenter{\xymatrix@R=0pt@C=0pt{a&b\\c&d}}\right),\pm\sqrt{c\tau+d}\right) $$
where $\left( \vcenter{\xymatrix@R=0pt@C=0pt{a&b\\c&d}}\right)\in\slz$ and $\sqrt{c\tau+d}$ is considered as a holomorphic function of $\tau$ in the upper half plane whose square is $c\tau+d$. The multiplication is defined so that the usual formulas for the transformation of modular forms of half integral weight work, which means that $$ (A,f(\tau))(B, g(\tau))=(AB,f(B(\tau))g(\tau))$$ for $A,B\in \slz$ and $f,g$ suitable functions on $\h$.

Suppose that $V$ is a vector space over $\Q$ and $(\;,\;)$ is a bilinear form on $V \times V$ with signature $(b^+,b^-)$. For an element $x\in V$ we will write $x^2:=(x,x)$ and $q(x)=\frac{1}{2}(x,x)$.   Let $L\subset V $ be a lattice. The dual lattice of $L$ is defined as $L'=\{x\in V| (x,L)\subset\Z \}$. We say that $L$ is even if $l^2\in2\Z$ for all $l\in L$. In this case $L$ is contained in $L'$ and $L'/L$ is a finite abelian group.

We let the elements $e_\nu$ for $\nu\in L'/L$ be the standard basis of the group
ring $\C[L'/L]$, so that $e_\mu e_\nu=e_{\mu+\nu}$. Recall that there is a unitary representation $\rho_L$ of the double cover $\mathrm{Mp}_2(\Z)$ of $\mathrm{SL}_2(\Z)$ on $\C[L'/L]$ defined by

\begin{align}\rho_L(\widetilde{T})(e_\nu)&=\e((\nu,\nu)/2)e_\nu \\
\rho_L(\widetilde{S})(e_\nu)&=i^{(b^-/2-b^+/2)}|L'/L|^{-1/2}\sum_{\mu\in
L'/L}\e(-(\mu,\nu))e_\mu,
\end{align}
where
\begin{equation}\label{generators}\widetilde{T}=\left(\left(\vcenter{\xymatrix@R=0pt@C=0pt{1&1\\0& 1}}\right),1\right),\mbox{and}\;\widetilde{S}=\left(\left(\vcenter{\xymatrix@R=0pt@C=0pt{0&1\\-1& 0}}\right),\sqrt{\tau}\right)\end{equation} are the standard generators of $\mathrm{Mp}_2(\Z)$.

A vector valued modular form of half-integral weight $k$ and representation $\rho_L$ is a function $f:\h\to \C[L'/L]$ that satisfies the following transformation law

$$f\left(\frac{a\tau+b}{c\tau+d}\right)=\sqrt{c\tau+d}^{2k}\rho_L\left(\left(\vcenter{\xymatrix@R=0pt@C=0pt{a&b\\c& d}}\right),\sqrt{c\tau+d}\right)f(\tau) .$$


We will use the notation  $\mathfrak{M}_k(\rho_L)$ for the space of real analytic,  $M_k(\rho_L)$ for the space of holomorphic,  $\widehat{M}_k(\rho_L)$ for the space of almost holomorphic,  and  $M^!_k(\rho_L)$ for the space of weakly holomorphic  modular forms of weight $k$ and representation $\rho_L$.

Let $M\subset L$ be a sublattice of finite index, then a vector valued modular form $f\in \mathfrak{M}_k(\slz,\rho_L)$
can be naturally viewed as a vector valued modular form in $f\in \mathfrak{M}_k(\slz,\rho_M)$. Indeed, we have the
inclusions $$M \subset L \subset L' \subset M'$$ and therefore
$$L/M \subset L'/M \subset M'/M.$$
We have the natural map $L'/M \to L'/L$, $\mu\to\bar{\mu}$.
\begin{lemma}\label{res/tr} There are two natural maps
$$\res_{L/M} : \mathcal{M}_k(\rho_L)\to \mathcal{M}_k(\rho_M),$$
and
$$\tr_{L/M} : \mathcal{M}_k(\rho_M) \to \mathcal{M}_k,(\rho_L).$$ For any $f\in A_k(\slz,\rho_L)$ and $g \in A_k(\slz,\rho_M)$ they are given as follows. For $\mu\in M'/M$ $$(\res_{L/M}(f))_\mu=\left\{\vcenter{\xymatrix@R=0pt@M=5pt{f_{\bar{\mu}},\;\mbox{if}\;\mu\in L'/M \\ 0,\;\mbox{if}\;\mu\notin L'/M}}\right., $$ for $\lambda\in L'/L$ $$(\tr_{L/M}(g))_\lambda=\sum_{\mu\in L'/M:\,\bar{\mu}=\lambda} g_\mu. $$
\end{lemma}
\section{Real analytic Jacobi forms}\label{Jacobi}
In this section we define  certain real-analytic functions similar to Jacobi forms.
Let $L$ be an even lattice of signature $(b^+,b^-)$. Let $v^+$ be a positive $b^+$ dimensional subspace of $L\otimes \R$. Denote by $v^-$ the orthogonal complement of $v^+$. For a vector $l\in L$ denote by $l_{v^+}$ and $l_{v^-}$ its projections on $v^+$ and $v^-$.

For $\lambda\in L'/L$ we define
$$\theta^{\mathrm{J}}_{L+\lambda}(\tau,z,v^+)=\sum_{l\in\lambda+L}\e(\tau l^2_{v^+}/2+\bar{\tau} l^2_{v^-}/2+(l,z)) $$
where $\tau\in\h$ and $z\in L\otimes\C$. It follows from Theorem 4.1 \cite{Bo1} that this function satisfies the following transformation properties
\begin{equation} \label{thetaj-1/tau} \theta^{\mathrm{J}}_{L+\lambda}
(\frac{-1}{\tau},\frac{z_{v^+}}{\tau}+\frac{z_{v^-}}{\bar{\tau}},v^+)=\end{equation}
$$i^{(b^-/2-b^+/2)}|L'/L|^{-1/2}\tau^{b^+/2}\bar{\tau}^{b^-/2}\e(\frac{z_{v^+}^2}{2\tau}+ \frac{z_{v^-}^2}{2\bar{\tau}})\sum_{\lambda_1\in L'/L}\e(-(\lambda,\lambda_1))\,\theta^{\mathrm{J}}_{\lambda_1+L}(\tau,z,v^+).$$
Define $$\Theta^{\mathrm{J}}_{L}(\tau,z,v^+):=\sum_{\lambda\in L'/L}\theta^{\mathrm{J}}_{L+\lambda}(\tau,z,v^+). $$

We have $$\Theta^{\mathrm{J}}_{L}(\frac{-1}{\tau},\frac{z_{v^+}}{\tau}
+\frac{z_{v^-}}{\bar{\tau}},v^+)=i^{(b^-/2-b^+/2)}\tau^{b^+/2}\bar{\tau}^{b^-/2}
\e(\frac{z_{v^+}^2}{2\tau}+ \frac{z_{v^-}^2}{2\bar{\tau}})\Theta^{\mathrm{J}}_{L}(\tau,z,v^+),$$
and for $m\in L'\;,n\in L$
$$\Theta^{\mathrm{J}}_{L}(\tau,z+\tau m+ n,v^+)=\e(-\tau m^2_{v^+}-\overline{\tau}m^2_{v^-}-(z,m))\Theta^{\mathrm{J}}_{L}(\tau,z,v^+). $$


\section{Regularized theta lift}\label{theta int}
In this section recall the definition of regularized theta lift given in  Borcherds's paper~\cite{Bo1}.

We let $M$ be an even lattice of signature $(2,b^-)$ with dual $M'$. The (positive) Grassmannian $G(M)$ is
the set of positive definite two dimensional subspaces $v^+$ of $M\otimes\R$. We write $v^-$ for the orthogonal complement of $v^+$, so that $M\otimes\R$ is the orthogonal direct sum of the positive definite subspace $v^+$ and the negative definite subspace $v^-$. The projection of $m\in M\otimes\R$ into a subspaces $v^+$ and $v^-$ is denoted by $m_{v^+}$ and  $m_{v^-}$ respectively, so that
$m=m_{v^+}+m_{v^-}$.

The vector valued  Siegel theta function $\Theta_M:\h\times G(M)\to\C[M'/M]$ of $M$ is defined by
$$\Theta_M(\tau,v^+)=y^{b^-/2}\sum_{m\in M'}\e(\tau m_{v^+}^2/2+\bar{\tau}m_{v^-}^2/2)e_{m+M}.$$
\begin{remark}Our definition of $\Theta_M$ differs from the one given in \cite{Bo1} by the multiple $y^{b^-/2}$.\end{remark}
Theorem 4.1 in \cite{Bo1} says that $\Theta_M(\tau,v^+)$ is a real-analytic vector valued modular form of weight $1-b^-/2$ and representation $\rho_M$ with respect to variable $\tau$.

We suppose that $f$ is some  $\C[M'/M]$-valued function
on the upper half plane $\h$ transforming under $\mathrm{SL}_2(\Z)$ with weight $1-b/2$
and representation $\rho_M$. Define
\begin{equation}\label{theta}\Phi_M(v^+,f):=\int_{\slz\backslash\h}\langle f(\tau),\overline{\Theta_M(\tau,v^+)}\rangle y^{-1-b^-/2}\,dx\,dy \end{equation}
(the product of $\Theta_M$ and $f$ means we take their inner product using
$(e_\mu,e_\nu)=1$ if $\mu=\nu$ and $0$ otherwise.)

The integral is often divergent and has to be regularized as follows. We
integrate over the region $\mathcal{F}_t$, where $$\mathcal{F}_\infty=\{\tau\in\h| -1/2<\Re(\tau)<1/2\;\mbox{and}\;|\tau|>1\} $$ is the
usual fundamental domain of $\slz$ and $\mathcal{F}_t$ is the subset of $\mathcal{F}_\infty$ of points $\tau$ with $\Im(\tau)<t$. Suppose that for $\Re(s)\gg 0 $ the limit $$\lim_{t\to\infty} \int
\limits_{F_t}f(\tau)\,\overline{\Theta_M(\tau,v^+)}\,y^{-1-b^-/2-s}\,dx\,dy$$ exists
and can be continued to a meromorphic function defined for all complex~$s$.
Then we define
$$\int_{\slz\backslash\h}^{\mathrm{reg}} f(\tau)\overline{\Theta_M(\tau,v^+)}y^{-1-b^-/2}\,dx\,dy$$ to be the constant term of the Laurent
expansion of this function at $s = 0$. 

Denote by $\mathrm{Aut}(M)$ the group of those isometries of $M\otimes\R$ that fix $M$.The action of $\mathrm{Aut}(M)$ on $f$ is given by action on $M'/M$. We define $\mathrm{Aut}(M,f)$ to be the subgroup of $\mathrm{Aut}(M)$ fixing $f$. The regularized integral $\Phi_M(v^+,f)$ is a function on the Grassmannian $G(M)$ that is invariant under $\mathrm{Aut}(M,f)$.

Suppose that $f\in \widehat{M}^{\,!}(\mpz,\rho_M)$ has a Fourier expansion
$$f_\mu(\tau)=\sum_{n\in\Q}\sum_{t\in\Z}c_\mu(n,t)\,\e(n,\tau)\,y^{-t} $$
and the coefficients $c_\mu(n,t)$ vanish whenever $m\ll 0$ or $t<0$ or $t\gg 0$.

We will say that a function $f$ has singularities of type $g$ at a point if $f-g$ can be redefined on a set of codimension at least 1 so that it becomes real analytic near the point.

Then the following theorems about regularized theta lift $\Phi_M(v^+,f)$ are proved in \cite{Bo1}.

\noindent{\sc Theorem B1}( Theorem 6.2 \cite{Bo1})\emph{ Near the point $v_0^+\in G(M)$, the function $\Phi_M(v^+,f)$ has a singularity of type $$\sum_{\xymatrix@R=0pt@C=0pt{\scriptstyle \lambda\in M'\cap v_0^-\\ \scriptstyle \lambda\neq 0}} -c_\lambda(\lambda^2/2,t)(-2\pi\lambda_{v^+}^2)^t  \log(\lambda_{v^+}^2)/t!.$$
In particular $\Phi_M$ is nonsingular ( real analytic) except along a locally finite set of codimension $2$ sub Grassmannians (isomorphic to $G(2,b^--1))$ of $G(M)$ of the form $\lambda^\bot$ for some negative norm vectors $\lambda\in M$.}

The open subset $$\mathcal{P}=\{[Z]\in\mathbb{P}(M\otimes\C); (Z,Z)=0\;\mbox{and}\; (Z,\overline{Z})>0\}$$
is isomorphic to $G(M)$ by mapping $[Z]$ to the subspace $\R\Re(Z)+\R\Im(Z)$.\\
We choose $m\in M,\;m'\in M'$ such that $m^2=0,\;(m,m')=1$. Denote $V_0:=M\otimes\Q\cap m^\bot\cap m'^\bot$. The tube domain \begin{equation}\label{mathcal{H}}\mathcal{H}=\{z\in V_0\otimes_\R \C| (\Im(z),\Im(z))>0\} \end{equation} is isomorphic to $\mathcal{P}$ by mapping $z\in\mathcal{H}$ to the class in $\mathbb{P}(M\otimes\C)$ of $$Z(z)=z+m'-\frac 12 ((z,z)+(m',m'))m. $$


 We consider the lattices $L=M\cap m^\bot$ and $K=(M\cap m^\bot)/\Z m$, and we identify $K\otimes\R$ with the subspace $M\otimes\R\cap m^\bot\cap m'^\bot$.

We write $N$ for the smallest positive value of the inner product $(m,l)$ with $l\in M$, so that $|M'/M|=N^2|K'/K|.$

Suppose that $f=\sum_\mu e_\mu f_{M+\mu}$ is a modular form of type $\rho_M$ and half integral weight $k$. Define a $\C[K'/K]$-valued function $$ f_K=\sum_{\kappa\in K'/K} e_\kappa f_{K+\kappa}(\tau)$$ by putting $$f_{K+\kappa}=\sum_{\xymatrix@R=0pt@C=0pt@M=1pt{\scriptstyle \mu\in M'/M:\\ \scriptstyle\mu|L=\kappa}} f_{M+\mu}(\tau) $$ for $\kappa\in K$.
The notation $\lambda|L$ means the restriction of $\lambda\in\mathrm{Hom}(M,\Z)$ to $L$, and $\gamma\in\mathrm{Hom}(K,\Z)$ is considered as an element of $\mathrm{Hom}(L,\Z)$ using the quotient map from $L$ to $K$. The elements of $M'$ whose restriction to $L$ is $0$ are exactly the integer multiples of $m/N$.

For $z\in\mathcal{H}$ denote by $w^+$ the positive definite subspace of $V_0$
$$w^+(z)=\R \Im(z)\in G(K). $$




Theorem 7.1 in \cite{Bo1} gives the Fourier expansion of the regularized theta lift and in the case when lattice $M$ has  signature $(2,b^-)$ it can be reformulated at the following form.\\
\noindent{\sc Theorem B2}\emph{ Let $M,K,m,m'$ be defined as above. Suppose
$$f =\sum_{\mu\in M'/M} e_\mu\sum_{m\in\Q} c_\mu (m,y)\e(mx)$$
is a modular form of weight $1-b^-$ and type $\rho_M$ with at most
exponential growth as $y \to \infty$. Assume that each function $c_\mu(m,y)\exp(-2\pi|m|y)$ has an
asymptotic expansion as $y \to \infty$ whose terms are constants times products of complex
powers of $y$ and nonnegative integral powers of  $\log(y)$. Let $z=u+iv$ be an element of a tube domain $\mathcal{H}$. If $(v,v)$ is sufficiently large then the
Fourier expansion of $\Phi_M(v^+(z),f)$ is given by the constant term of the Laurent expansion
at $s = 0$ of the analytic continuation of
\begin{equation}\label{Fourier}\sqrt{q(v)}\Phi_K(w^+(z),f_K)+
\frac{1}{\sqrt{q(v)}}\sum_{l\in K'}\sum_{\xymatrix@C=0pt@R=0pt@M=1pt{\scriptstyle \mu\in M'/M :\\ \scriptstyle \mu|L=l}}\sum_{n>0}\e((nl,u-m')+(n\mu,m'))\times\end{equation}
\begin{equation*}\times\int_{y>0}c_\mu(l^2/2,y)\exp(-\frac{\pi n^2 v^2}{2 y}-\pi y (\frac{2(l,v)^2}{v^2}-l^2))y^{-s-3/2}dy \end{equation*}
(which converges for $\Re(s) \gg 0$ to a holomorphic functions of s which can be analytically
continued to a meromorphic function of all complex $s$).}
Lattice $K$ has signature $(1,b^--1)$, so  $G(K)$ is real hyperbolic space of dimension $b^--1$ and the
singularities of $\Phi_K$ lie on hyperplanes of codimension 1. Then the set of points where $\Phi_K$
is real analytic is not connected.  The
components of the points where $\Phi_K$ is real analytic are called the Weyl chambers of $\Phi_K$.
 If $W$ is a Weyl
chamber and $l\in K$ then $(l,W) > 0$ means that $l$ has positive inner product with all
elements in the interior of $W$.

\section{Infinite products}\label{S borch prod}


There is a principal $\C^\ast$ bundle $\mathcal{L}$ over the hermitian
symmetric space $\mathcal{H}$, consisting of the norm $0$ points $Z = X +iY\in M\otimes
\C$ such that $X$
and $Y$ form an oriented base of an element of $G(M)$. The fact that $Z = X + iY$
has norm $0$ is equivalent to saying that $X$ and $Y$ are orthogonal and have the same
norm. We define an automorphic form of weight $k$ on $G(M)$ to be a function $\Psi$ on $\mathcal{L}$
which is homogeneous of degree $-k$ and invariant under some subgroup $\Gamma$ of finite index of
$\mathrm{Aut}(M)^+$. More generally if $\chi$ is a one dimensional representation of $\Gamma$ then we say $\Psi$	 is an
automorphic form of character $\chi$ if $\Psi(\sigma(Z)) = \chi(\sigma)\Psi	(Z)$ for $\sigma\in\Gamma$.\\






Suppose that $f\in M^{\,!}_{1-b^-/2}(\slz,\rho_M)$ has a Fourier expansion $$f(\tau)=\sum_{\lambda\in M'/M}\sum_{n\gg-\infty}c_\lambda (n)e(n\tau)e_\lambda. $$

\noindent{\sc Theorem B3}(\cite{Bo1}, Theorem 13.3)\emph{ Suppose that $f\in M^{\,!}_{1-b^-/2}(\slz,\rho_M)$ is holomorphic on $\h$ and meromorphic at cusps and whose Fourier coefficients $c_\lambda (n)$ are integers for $n\leq 0$. Then there is a meromorphic function $\Psi_M(Z,f)$ on $\mathcal{L}$ with the following properties.}\begin{description}\item{1.} \emph{$\Psi$ is an automorphic form of weight $c_0(0)/2$ for the group $\mathrm{Aut}(M,f)$ with respect to some unitary character of $\mathrm{Aut}(M,f)$\item{2.} The only zeros and poles of $\Psi_M$ lie on the rational quadratic divisors $l^\bot$ for $l\in M$, $l^2<0$ and are zeros of order $$\sum_{\xymatrix@C=0pt@R=0pt@M=1pt{\scriptstyle x\in\R^+\,:\\ \scriptstyle xl\in M'}} c_{xl}(x^2l^2/2) $$ \item{3.} $$ \Phi_M(Z,f)=
-4\log|\Psi_M(Z,f)|-2c_0(0)(\log|Y|+\Gamma'(1)/2+\log\sqrt{2\pi}).$$
\item{4.} For each primitive norm $0$ vector $m$ of $M$ and for each Weyl chamber $W$ of $K$ the restriction $\Psi_m(Z(z),f)$ has an infinite product expansion converging when $z$ is in a neighborhood of the cusp of $m$ and $Im(z)\in W$ which is some constant of absolute value}
    $$\prod_{\xymatrix@R=0pt@C=0pt@M=1pt{\scriptstyle\delta\in \Z/N\Z \\ \scriptstyle \delta\neq 0}} (1-\e(\delta/N))^{c_{\delta m/N}(0)/2} $$
    \emph{times}
$$\e((Z,\rho(K,W,f_{K})))\prod_{\xymatrix@C=0pt@R=0pt@M=1pt{\scriptstyle k\in K':\\ \scriptstyle (k,W)>0}}\prod_{\xymatrix@C=0pt@R=0pt@M=1pt{\scriptstyle \mu\in M'/M :\\ \scriptstyle \mu|L=k}} (1-\e((k,Z)+(\mu,m')))^{c_\mu(k^2/2)}.$$
Here the vector $\rho(K,W, f_K)$ is the Weyl vector, which can be evaluated explicitly
using the theorems in Section 10 of \cite{Bo1}.
\end{description}
\begin{remark}In the case then $M$ has no primitive norm 0 vectors Fourier expansions of $\Psi$ do not exist.\end{remark}


\section{A see-saw identity}
 In the paper \cite{Ku84} S. Kudla  introduced the notion of a ``see-saw dual reductive pair''. It gives rise in a formal way to a family of identities between inner products of automorphic forms on different groups, thus clarifying  the source of identities of this type which appear  in many places in the literature, often obtained from complicated manipulations. In this section we prove a see-saw identity for the regularized theta integrals described in previous section.

Suppose that $L\subset V $ is an even lattice of signature $(2,b)$. Let $V=V_1\oplus V_2$ be the rational orthogonal splitting of $(V,q)$ such that the space $V_1$ has the signature $(2,b-d)$ and the space $V_2$ has the signature $(0,d)$. Consider two lattices $N:=L\cap V_1$ and $M=L\cap V_2$. We have two orthogonal projections
 $$ \mathrm{pr}_M:L\otimes\R\to M\otimes\R \;\mbox{and}\;\mathrm{pr}_N:L\otimes\R\to N\otimes\R.$$
 Let $M'$ and $N'$ be the dual lattices of $M$ and $N$. 
 We have the following  inclusions
 $$ M\subset L,\;N\subset L,\;M\oplus N\subseteq L\subseteq L'\subseteq M'\oplus N', $$ and equalities of the sets $$\mathrm{pr}_M(L')=M',\;\mathrm{pr}_N(L')=N'. $$

  Consider a rectangular $|L'/L|\times|N'/N|$ dimensional matrix $T_{L,N}$ with entries $$\vartheta_{\lambda,\nu}(\tau)=\sum_{\xymatrix@R=0pt@C=0pt@M=1pt{\scriptstyle m\in M': \\ \scriptstyle  m+\nu\in\lambda+L}}\e(-\tau m^2/2)$$ where $ \lambda\in L'/L, \nu\in N'/N, \tau\in\h .$
  This sum is well defined since $N\subset L$.
  Note that the lattice $M$ is negative definite and hence the  series converge.
  \begin{lemma}\label{l1} For $\lambda\in L'/L$ and $\nu\in N'/N$ the function $\vartheta_{\lambda,\nu}$ satisfies the following transformation properties
  \begin{equation}\label{vartau+1}\vartheta_{\lambda,\nu}(\tau+1)=
  \e(\nu^2/2-\lambda^2/2)\vartheta_{\lambda,\nu}(\tau),\end{equation}
  \begin{equation}\label{var-1/tau}|L'/L|^{-1/2}\sum_{\lambda_1\in L'/L}\e(-(\lambda,\lambda_1))\,\vartheta_{\lambda_1,\nu}(-1/\tau)=\end{equation} $$(\tau/i)^{\,d/2}\,|N'/N|^{-1/2}\sum_{\nu_1\in N'/N}\e(-(\nu,\nu_1))\,\vartheta_{\lambda,\nu_1}(\tau).$$

   \end{lemma}
  \begin{proof}For $\mu\in M'/M$ denote
  $$ \theta_\mu(\tau):=\sum_{m\in\mu+M}\e(-m^2/2\tau).$$ We have \begin{equation}\label{vartheta} \vartheta_{\lambda,\nu}(\tau)=\sum_{\xymatrix@C=0pt@R=0pt@M=1pt{\scriptstyle \mu\in M'/M :\\ \scriptstyle \mu+\nu\in\lambda+L}}\theta_\mu(\tau). \end{equation} Here $\mu +\nu$ is considered as an element in $M'\oplus N'/M\oplus N$ and $\lambda+L$ is a subset of $M'\oplus N'/M\oplus N$. Suppose that $\lambda\in L'/L, \mu \in M'/M, \nu \in N'/N$ satisfy the condition $\mu+\nu\in\lambda+L$. Then
  $$\e(\mu^2/2+\nu^2/2)=\e(\lambda^2/2), $$ and hence
  $$\theta_\mu(\tau+1)=\e(-\mu^2/2)\,\theta_\mu(\tau)= \e(\nu^2/2-\lambda^2/2)\,\theta_\mu(\tau).$$
  Thus, equation \eqref{vartau+1} follows from the previous formula and \eqref{vartheta}.

  Now we prove equation \eqref{var-1/tau}. 
Suppose that $z_N\in N\otimes\C$ and $v^+\in G(N)\hookrightarrow G(L)$. Then the real analytic Jacobi theta function defined in Section \ref{Jacobi} satisfies
  $$ \theta^{\mathrm{J}}_{L+\lambda}(\tau,z_N,v^+)=\sum_{\nu\in N'/N}\theta^{\mathrm{J}}_{N+\nu}(\tau,z_N,v^+)
  \overline{\vartheta_{\lambda,\nu}(\tau)}.$$
  Transformation property \eqref{var-1/tau} follows from the above equation and  transformation property \eqref{thetaj-1/tau} of functions $\theta^{\mathrm{J}}_{L+\lambda}$ and $\theta^{\mathrm{J}}_{N+\nu}$.
  \end{proof}


\begin{theorem}Suppose that the lattices $L$, $M$ and $N$ are defined as above. Then  there is a map $T_{L,N}:M_k(\mathrm{Mp_2(\Z)},\rho_L)\to M_{k+d/2}(\mathrm{Mp_2(\Z)},\rho_N)$ sending a function $f=(f_\lambda)_{\lambda\in L'/L}$ to $g=(g_\nu)_{\nu\in N'/N}$ defined as
$$g_\nu(\tau)=\sum_{\lambda\in L'/L}\vartheta_{\lambda,\nu}(\tau)\,f_\lambda(\tau).$$
In other words $$g=T_{L,N}f $$
where $f$ and $g$ are considered as column vectors.
\end{theorem}
\begin{proof}Follows from Lemma \ref{l1}.\end{proof}

\begin{theorem}\label{Th see-saw}Let $L$, $M$, $N$ be as above. Denote by  $i:G(N)\to G(L)$  a natural embedding induced by inclusion $N\subset L$.  Then, for $v^+\in G(N)$ and the theta lift of a function $f\in \widehat{M}_{1-b/2}^{\,!}(\slz,\rho_L)$  the following holds \begin{equation} \label{see-saw}\Phi_L(i(v^+),f)=\Phi_N(v^+,\theta_{L,N}(f)).\end{equation}
\end{theorem}
\begin{proof}
For a vector $l\in L'$ denote $m=\mathrm{pr}_M(l)$ and $n=\mathrm{pr}_N(l)$. Recall that $m\in M'$ and $n\in N'$. Since $v^+$ is an element of $G(N)$  it is orthogonal to $M$. We have $$l^2_{v^+}=n^2_{v^+},\;l^2_{v^-}=m^2+n^2_{v^-}.$$ Thus for $\lambda\in L'/L$ we obtain $$\Theta_{\lambda+L}(\tau,v^+)=\sum_{l\in\lambda+L}\e(\tau l^2_{v^+}/2+\bar{\tau} l^2_{v^-}/2)= $$
$$\sum_{\xymatrix@R=0pt@C=0pt@M=1pt{\scriptstyle m\in M',\, n\in N': \\ \scriptstyle m+n\in\lambda+L}}\e(\tau n^2_{v^+}/2+\bar{\tau} n^2_{v^-}/2+\bar{\tau}m^2/2). $$
Since $N\subset L$ we can rewrite this sum as
$$ \Theta_{\lambda+L}(\tau,v^+)=\sum_{\nu\in N'/N}\Theta_{\nu+N}(\tau,v^+)\overline{\vartheta_{\nu,\lambda}(\tau)}.$$
Thus, we see that for $f=(f_\lambda)_{\lambda\in L'/L}$ the following scalar products are equal
$$\langle f,\overline{\Theta_L(\tau,v^+)}\rangle =\langle T_{L,N}(f),\overline{\Theta_N(\tau,v^+)}\rangle .$$ So, the regularized integrals \eqref{theta} of both sides of the equality are also equal.  \end{proof}

\noindent {\bf Remark.} Theorem \ref{Th see-saw} works even in the case when $v^+$ is a singular point of $\Phi_L(v^+,f)$.  If the constant terms of $f$ and $T_{L,N}(f)$ are different, then subvariety $G(N)$ lies in singular locus of  $\Phi_L(v^+,f)$. On the other hand, if constant terms of $f$ and $T_{L,N}(f)$ are equal then, singularities cancel at the points of $G(N)$.
\section{Lattice $M_2(\Z)$}\label{m2z}
Consider the lattice of integral $2\times 2$ matrices denoted by $M_2(\Z)$. Equipped with a quadratic form  $q(x):=-\det x$ it becomes an even unimodular lattice (in our notations $x^2=2q(x)$).

The Grassmannian $G(M_2(\Z))$ turns out to be isomorphic to $\h\times\h$. This isomorphism can be constructed in the following way. For the pair of points $(\tau_1,\tau_2)\in\h\times\h$ consider the element of the norm zero $$Z=\left(\vcenter{\xymatrix@R=0pt@C=0pt{\tau_1\tau_2&\tau_1\\ \tau_2& 1}}\right)\in M_2(\Z)\otimes \C.$$ Define $v^+(\tau_1,\tau_2)$ be the vector subspace of $M_2(\Z)\otimes\R$ spanned by two vectors $X=\Re(Z)$ and $Y=\Im(Z)$.
The group $\slz\times\slz$ acts on $M_2(\Z)$ by $(\gamma_1,\gamma_2)(x)=\gamma_1 x\gamma_2^t$ and preserves the norm.The action of $\slz\times\slz$ on the Grassmannian agrees with the action on $\h\times\h$ by fractional linear transformations
$$(\gamma_1,\gamma_2)(v^+(\tau_1,\tau_2))=v^+(\gamma_1(\tau_1),\gamma_2(\tau_2)). $$

We have $$ X^2=Y^2=\frac 12 (Z,\overline{Z})=-\frac 12 (\tau_1-\bar{\tau_1})(\tau_2-\bar{\tau_2}),$$
$$(X,Y)=Z^2=0. $$
For $l=\left( \vcenter{\xymatrix@R=0pt@C=0pt{a&b\\c&d}}\right)\in M_2(\Z)$ and $v^+=v^+(\tau_1,\tau_2)$ we have
$$ l^2_{v^+}/2=\frac{(l,Z)(l,\overline{Z})}{(Z,\overline{Z})}
=\frac{|d\tau_1\tau_2-c\tau_1-b\tau_2+a|^2}
{-(\tau_1-\bar{\tau}_1)(\tau_2-\bar{\tau}_2)}.$$
Denote $$\Theta(\tau;\tau_1,\tau_2):=\Theta_{M_2(\Z)}(\tau,v^+(\tau_1,\tau_2)) $$
where $\tau=x+iy$. Considered as a function of $\tau$ $\Theta$ belongs to $\mathfrak{M}_0(\slz)$ and we can  explicitly write this  function as
$$\Theta(\tau;\tau_1,\tau_2)=y\sum_{a,b,c,d\in\Z} \e\left( \frac{|a\tau_1\tau_2+b\tau_1+c\tau_2+d|^2}{-(\tau_1-\bar{\tau}_1)
(\tau_2-\bar{\tau}_2)}(\tau-\bar{\tau})-(ad-bc)\bar{\tau}\right)$$

$$=y\sum_{a,b,c,d\in\Z} \e\left( \frac{|a\tau_1\tau_2+b\tau_1+c\tau_2+d|^2}{-(\tau_1-\bar{\tau}_1)
(\tau_2-\bar{\tau}_2)}\tau-\frac{|a\tau_1\bar{\tau}_2+b\tau_1+c\bar{\tau}_2+d|^2}
{-(\tau_1-\bar{\tau}_1)(\tau_2-\bar{\tau}_2)}\bar{\tau}\right). $$

 \section{Higher Green's functions as theta lifts}

The key point of our proof is the following observation
\begin{proposition}\label{key} Denote by $\Delta^z$ the hyperbolic Laplacian with respect to variable $z$. For the function $\Theta$ defined in previous section the following identities hold
$$\Delta^\tau \Theta(\tau;\tau_1,\tau_2)= \Delta^{\tau_1} \Theta(\tau;\tau_1,\tau_2)=\Delta^{\tau_2} \Theta(\tau;\tau_1,\tau_2).$$
\end{proposition}
A similar identity can be found in \cite{Br}.

Suppose that $\bm{\lambda}=\{\lambda_m\}_{m=1}^\infty$ is a relation on $S_{2k}(\slz)$ (definition is given at the introduction). Then there exists a unique weakly holomorphic modular form $g_{\bm{\lambda}}$ of weight $2-2k$ with Fourier expansion of the form $$ \sum _{m}\lambda_m\, q^{-m}+O(1).$$

Define curves $T_m\subset \h\times\h$.
For a relation $\bm{\lambda}$ consider a divisor $$D_{\bm{\lambda}}:=\sum_m\lambda_m T_m. $$
Denote by $S_{\bm{\lambda}}$ the support of $D_{\bm{\lambda}}$. It follows from the properties (i), (iv) of Green's function given at the introduction that the singular locus of $G_{k,\bm{\lambda}}$ is equal to $S_{\bm{\lambda}}$.

Consider the function $h_{\bm{\lambda}}:=R^{k-1} (g_{\bm{\lambda}})$ which belongs to $\widehat{M}^{\,!}_0(\slz)$.
\begin{theorem}\label{Th g} The following identity holds \label{g} $$G_{k,\bm{\lambda}}(\tau_1,\tau_2)= \Phi_{M_2(\Z)}(v^+(\tau_1,\tau_2),h_{\bm{\lambda}}).$$ Here \begin{equation} \label{limFt} \Phi_{M_2(\Z)}(v^+(\tau_1,\tau_2),h_{\bm{\lambda}})= \lim_{t\to\infty}\int\limits_{F_t} h_{\bm{\lambda}}(\tau)\,\overline{\Theta(\tau;\tau_1,\tau_2)}\,y^{-2}dx\,dy.\end{equation} \end{theorem}
\begin{proof}
We verify that the function $\Phi_{M_2(\Z)}(v^+(\tau_1,\tau_2),h_{\bm{\lambda}})$ satisfy conditions (i)-(iv) listed at the introduction.

 It follows from Theorem B1 (Theorem 6.2 \cite{Bo1} p. 24) that the limit \eqref{limFt} exists for all $\tau_1,\tau_2\in\h\times\h\setminus S_{\bm{\lambda}}$, moreover, it defines a real-analytic function on this set. For the convenience of the reader we repeat the argument given in \cite{Bo1}. The function $h_{\bm{\lambda}}$ has the Fourier expansion $$h_{\bm{\lambda}}(\tau)=\sum_{\xymatrix@R=0pt@C=0pt@M=2pt{\scriptstyle n\in\Z\\ \scriptstyle n\gg -\infty}} c(n,y)\e(n\tau). $$  Fix $v^+=v^+(\tau_1,\tau_2)$ for some $\tau_1,\tau_2\in\h\times\h$. 
  For $t>1$ the set $F_t$ can be decomposed into two parts $F_t=F_1\cup \Pi_t$ where $\Pi_t$ is a rectangle $\Pi_t=[-1/2,1/2]\times[1,t]$. It suffices to show that the limit
 $$\lim_{t\to\infty}\int\limits_{\Pi_t} h_{\bm{\lambda}}(\tau)\,\overline{\Theta_{M_2(\Z)}(\tau;v^+)}\,y^{-2}dx\,dy $$ exists for all $(\tau_1,\tau_2)\notin S_{\bm{\lambda}}$. It can be seen from the following computation
 \begin{align*}&\int\limits_{\Pi_t} h_{\bm{\lambda}}(\tau)\,\overline{\Theta_{M_2(\Z)}(\tau;v^+)}\,y^{-2}dx\,dy= \\
 &\int\limits_{\Pi_t} \sum_{n\in\Z}\sum_{l\in M_2(\Z)} c(n,y)\e(n\tau)\overline{\e(\tau l_{v^+}^2/2+\bar{\tau}l_{v^-}^2/2)} y^{-1}dx\,dy=\\
 &\int\limits_{-1/2}^{1/2}\int\limits_1^t \sum_{n\in\Z}\sum_{l\in M_2(\Z)} c(n,y)\e((n-l^2/2)x)\exp(-2\pi y l_{v^+}^2) y^{-1}\,dx\,dy= \\
 &\int\limits_1^t \sum_{l\in M_2(\Z)} c(l^2/2,y) \exp(-2\pi y l_{v^+}^2) y^{-1}\,dy. \end{align*}

Properties (i) and (iv) follow from Theorem B1 (Theorem 6.2 \cite{Bo1} p. 24).

Property (ii) is obvious since the function $\Theta(\tau;\tau_1,\tau_2)$ is $\slz$-invariant in variables $\tau_1$ and $\tau_2$.

Property (iii) formally follows from Proposition \ref{key}.
We have
$$\Delta^{\tau_1}\Phi_{M_2(\Z)}(h_{\bm{\lambda}},v^+(\tau_1,\tau_2))=\lim_{t\to\infty}\int\limits_{F_t} h_{\bm{\lambda}}(\tau)\overline{\Delta^{\tau_1}\Theta(\tau;\tau_1,\tau_2)}y^{-2}dx\,dy. $$
Using Proposition \ref{key} we arrive at $$\Delta^{\tau_1}\Phi_{M_2(\Z)}(h_{\bm{\lambda}},v^+(\tau_1,\tau_2))=\lim_{t\to\infty}\int\limits_{F_t} h_{\bm{\lambda}}(\tau)\overline{\Delta^{\tau}\Theta(\tau;\tau_1,\tau_2)}y^{-2}dxdy. $$
It follows from Stokes' theorem that $$\int\limits_{F_t} h_{\bm{\lambda}}(\tau)\overline{\Delta^{\tau}\Theta(\tau;\tau_1,\tau_2)}y^{-2}dxdy-\int\limits_{F_t} \Delta h_{\bm{\lambda}}(\tau)\overline{\Theta(\tau;\tau_1,\tau_2)}y^{-2}dxdy= $$ $$\left.\int\limits_{-1/2}^{1/2} (h_{\bm{\lambda}}\,\overline{L_0(\Theta)}-L_0(h_{\bm{\lambda}})\,\overline{\Theta})y^{-2}dx\right|_{y=t}. $$
This expression tends to zero as $t$ tends to infinity.
Since $g_{\bm{\lambda}}\in F_{k,2-2k}$ it follows from Proposition \ref{diff} that  $\Delta h_{\bm{\lambda}}= k(1-k) h_{\bm{\lambda}}$. Thus, we see that theta lift $\Phi_{M_2(\Z)}(h_{\bm{\lambda}},v^+)$ satisfies the desired differential equation $$\Delta^{\tau_i}\Phi_{M_2(\Z)}(h_{\bm{\lambda}},v^+(\tau_1,\tau_2))=k(1-k)\Phi_{M_2(\Z)}(h_{\bm{\lambda}},v^+(\tau_1,\tau_2)),\;i=1,2.$$

It remains to prove (v). We can compute the Fourier expansion of $\Phi_{M_2(\Z)}(v^+(\tau_1,\tau_2),h_{\bm{\lambda}})$ using Theorem B2.
 We select a primitive norm zero vector  $m:=\left( \vcenter{\xymatrix@R=0pt@C=0pt{1&0\\0&0}}\right)\in M$ and choose $m':=\left( \vcenter{\xymatrix@R=0pt@C=0pt{0&0\\0&1}}\right)$ so that $(m,m')=1$. For this choice of vectors $m,\;m'$ the tube domain $\mathcal{H}$ defined by equation \eqref{mathcal{H}} is isomorphic to $\h\times\h$ and the map between $\h\times\h$ and Grassmanian  $G(M_2(\Z)$ is given by $$(\tau_1,\tau_2)\to v^+(\tau_1,\tau_2). $$
The lattice $K=(M\cap m^\bot)/m$ can be identified with $$M\cap m^\bot\cap m'^\bot=\left\{\left.\left( \vcenter{\xymatrix@R=0pt@C=0pt{0&b\\c&0}}\right)\right|b,c\in\Z\right\}.$$
 Denote $x_i=\Re(\tau_i)$ and $y_i=\Im(\tau_i)$ for $i=1,2$. The subspace $w^+(\tau_1,\tau_2)\in G(K)$ is equal to $$\R\left( \vcenter{\xymatrix@R=0pt@C=0pt{0&y_1\\y_2&0}}\right). $$ 
Suppose that $$h_{\bm{\lambda}}(\tau)=\sum_{n\in\Q} c (n,y)\,\e(n\tau)$$
where $$c (n,y)=\sum_{t\geq 0}b(n,t)y^{-t}. $$ 
We can rewrite equality \eqref{Fourier} as
\begin{equation}\label{Fourier1} \Phi_M(v^+,h_{\bm{\lambda}})=\frac{1}{\sqrt{2}|m_{v^+}|}\Phi_K(w^+,h)+
\frac{\sqrt{2}}{|m_{v^+}|}\sum_{l\in K}\sum_{n>0}\e((nl,m_v''))\times\end{equation}
\begin{equation*}\times\int\limits_0^\infty c(l^2/2,y)\exp(-\pi n^2/2y m^2_{v^+}-2\pi y l_{w^+}^2)y^{-3/2}dy =\end{equation*}

$$=\sqrt{y_1 y_2}\,\Phi_K(w^+(\tau_1,\tau_2),f_K)+
\frac{1}{\sqrt{y_1 y_2}}\sum_{l\in K}\sum_{n>0}\e((nl,u))\times$$
\begin{equation*}\times\int\limits_0^\infty c(l^2/2,y)\exp(-\frac{\pi n^2 y_1 y_2}{ y}-\pi y \frac{(l,v)^2}{y_1y_2})y^{-3/2}dy, \end{equation*}
where $u=\Re\left( \vcenter{\xymatrix@R=0pt@C=0pt{0&\tau_1\\ \tau_2&0}}\right)$ and $v=\Im\left( \vcenter{\xymatrix@R=0pt@C=0pt{0&\tau_1\\ \tau_2&0}}\right)$.
 We choose a primitive norm $0$ vector $r=\left( \vcenter{\xymatrix@R=0pt@C=0pt{0&1\\0&0}}\right)\in K$. It follows from Theorem 10.2 \cite{Bo1} that \begin{equation}\label{PhiK} \Phi_K(w^+,h_{\bm{\lambda}})=\sum_t b(0,t) (2 r_{w^+}^2)^{t+1/2} \pi^{-t-1}\Gamma(t+1)(-2\pi i)^{2t+2} B_{2t+2}/(2t+2)!=$$
 $$ \sum_t b(0,t) (y_2/y_1)^{t+1/2} \pi^{-t-1}\Gamma(t+1)(-2\pi i)^{2t+2} B_{2t+2}/(2t+2)!.\end{equation}
 In the case $l_{w^+}\neq 0$ it follows from Lemma 7.2
 \begin{equation}\label{neq0}\int_{y>0}c(l^2/2,y)\exp(-\pi n^2/2y m^2_{v^+}-\pi y l_{w^+}^2)y^{-3/2}dy=\end{equation}
 $$\sum_t 2b(l^2/2,t)(2|m_{v^+}|\,|l_{w^+}|/n)^{t+1/2} K_{-t-1/2}(2\pi n|l_{w^+}|/|m_{v^+}|). $$
 In case $l_{w^+}= 0$ it follows from Lemma 7.3
\begin{equation}\label{eq0}\int_{y>0}c(l^2/2,y)\exp(-\pi n^2/2y m^2_{v^+}-2\pi y l_{w^+}^2)\,y^{-3/2}\,dy=\end{equation}
 $$\sum_t b(l^2/2,t)(2m_{v^+}^2/\pi n^2)^{t+1/2} \Gamma(t+1/2). $$
 Substituting formulas \eqref{PhiK}-\eqref{eq0} into \eqref{Fourier1} we obtain
 \begin{equation}\label{Fourier2}\Phi_M(v^+(\tau_1,\tau_2),h_{\bm{\lambda}})=\end{equation} \begin{align*}-\sum_t &\frac{y_2^{t+1}}{y_1^t} b(0,t)(-4\pi)^{t+1}\zeta(-2t-1)\frac{t!}{(2t+1)!}+ \\4\sum_t&(y_1 y_2)^{-t} b(0,t)(4\pi)^{-t}\zeta(2t+1)\frac{2t!}{t!}+\\
 4\sum_t&\sum_{\xymatrix@C=0pt@R=0pt@M=1pt{\scriptstyle (c,d)\in \Z^2 \\ \scriptstyle (c,d)\neq(0,0)}} \sum_{n>0}(y_1 y_2)^{-t} b(cd,t)n^{-2t-1}\times \\
 &\times\e(ncx_1+ndx_2)|ncy_1+ndy_2|^{\,t+1/2}K_{-t-1/2}(2\pi|ncy_1+ndy_2|).\end{align*}
 We see from \eqref{Fourier2} that $\Phi_M(v^+(\tau_1,\tau_2),h_{\bm{\lambda}})\to 0$ as $y_1\to \infty$. This finishes the proof.
\end{proof}

Now we can analyze the CM values of $G_{k,\bm{\lambda}}$ using the see-saw identity \eqref{see-saw} .

Let $\tau_1,\tau_2\in\h$ be two CM points in the same quadratic imaginary field $\Q(\sqrt{-D})$. Let $v^+(\tau_1,\tau_2)$ be a two dimensional  positive definite subspace of $M_2(\R)$ defined as  \begin{equation}\label{v+(tau1,tau2)}v^+(\tau_1,\tau_2)=\R\Re\left(\vcenter{\xymatrix@R=0pt@C=0pt{\tau_1\tau_2&\tau_1\\ \tau_2& 1}}\right)+\R \Im \left(\vcenter{\xymatrix@R=0pt@C=0pt{\tau_1\tau_2&\tau_1\\ \tau_2& 1}}\right).\end{equation} In the case when $\tau_1$ and $\tau_2$ lie in the same quadratic imaginary field the subspace $v^+(\tau_1,\tau_2)$ defines a rational splitting of $M_2(\Z)\otimes\Q$. So, we can consider two lattices $N:=v^+(\tau_1,\tau_2)\cap M_2(\Z)$ and $M:=v^-(\tau_1,\tau_2)\cap M_2(\Z)$.

 The Grassmannian $G(N)$ consists of a single point $N\otimes \R$ and its image in $G(M_2(\Z))$ is $v^+(\tau_1,\tau_2)$.

   Since the lattice $N$ has signature $(2,0)$ the theta lift of a function $f\in \widehat{M}^{\,!}_1(\slz,\rho_N)$ is just a number and it is equal to a regularized integral $$\Phi_N(f)=\int^{\mathrm{reg}}_{\slz\backslash\h} \langle f(\tau),\overline{\Theta_N (\tau)}\rangle\, y^{-1} \,dx\,dy. $$ Here ${\Theta}_N $ is a usual (vector valued) theta function of the lattice $N$.
  The matrix  $T_{M_2(\Z),N}=(\vartheta_{0,\nu})_{\nu\in N'/N}$ becomes a vector in this case and it is given by
$$ \vartheta_{0,\nu}(\tau)=\sum_{m\in M'\cap(-\nu+M_2(\Z))}\e(-\tau m^2/2).$$
Till the end of this section we will simply write $\vartheta_{\nu}(\tau)$ for $\vartheta_{0,\nu}(\tau)$.
\begin{theorem}\label{Th f} Suppose that two  CM-points $\tau_1,\tau_2$ and a lattice $N\subset M_2(\Z)$ are defined as above. Let $\bm{\lambda}$ be a relation for $S_{2k}(\slz)$ and let $g_{\bm{\lambda}}\in M^{\,!}_{2-2k}(\slz)$ be the corresponding weakly holomorphic form defined in Proposition \ref{g_lambda}. Then, if $(\tau_1,\tau_2)\notin S_{\bm{\lambda}}$  we have $$\Phi_{M_2(\Z)}(v^+(\tau_1,\tau_2),R^{k-1} (g_{\bm{\lambda}}))=\Phi_{N}(f),$$ where $f=(f_\nu)_{\nu\in N'/N}\in M^{\,!}_1(\slz,\rho_N)$ is given by $$f_\nu=[\,g_{\bm{\lambda}},\vartheta_{\nu}]_{k-1}.$$\end{theorem}\begin{proof}

For $(\tau_1,\tau_2)\notin S_{\bm{\lambda}}$ the constant term (with respect to $\e(x)$) of the product $\langle R^{k-1}(g_{\bm{\lambda}})(\tau),\overline{\Theta(\tau;\tau_1,\tau_2)}\rangle $ is equal to $$ \sum_{l\in M_2(\Z)} a_{l^2/2}(y) \exp(-2\pi y l_{v^+}^2) y$$ and decays as $O(y^{2-k})$ as $y\to\infty$. Thus, $$\Phi_{M_2(\Z)}(v^+(\tau_1,\tau_2),R^{k-1} (g_{\bm{\lambda}}))=\lim_{t\to\infty}\int\limits_{F_t} R^{k-1} (g_{\bm{\lambda}})(\tau)\,\overline{\Theta(\tau;\tau_1,\tau_2)}\,y^{-2}dx\,dy. $$
It follows from the see-saw identity \eqref{see-saw}
$$\Phi_{M_2(\Z)}(v^+(\tau_1,\tau_2),R^{k-1} (g_{\bm{\lambda}}))= \lim_{t\to\infty}\int\limits_{F_t}\langle R^{k-1} (g_{\bm{\lambda}})\,\vartheta,\overline{\Theta}_N\rangle y^{-1} \,dx\,dy.$$
By Proposition \ref{RC1}

\begin{equation}\label{R(g)vartheta}R^{k-1} (g_{\bm{\lambda}})\,\vartheta_\nu =(-1)^{k-1}[g_{\bm{\lambda}},\vartheta_\nu]_{k-1}+R(\sum_{s=0}^{k-2} b_{s} R^{s}(g_{\bm{\lambda}})\,R^{k-2-s}(\vartheta_\nu))\end{equation}
where  $b_s$ are some rational numbers.
For $\nu\in N'/N$ denote $$\psi_\nu(\tau):= \sum_{s=0}^{k-2} b_{s} R^{s}(g_{\bm{\lambda}})\,R^{k-2-s}(\vartheta_\nu).$$
Using \eqref{R(g)vartheta} we can write
$$\lim_{t\to \infty} \int\limits_{F_t} \langle R^{k-1} (g_{\bm{\lambda}})\,\vartheta,\overline{\Theta}_N\rangle y^{-1} \,dx\,dy=$$ $$(-1)^{k-1}\,\lim_{t\to \infty} \int\limits_{F_t}  \langle [g_{\bm{\lambda}},\vartheta]_{k-1},\overline{\Theta}_N\rangle y^{-1} \,dx\,dy+ \lim_{t\to \infty} \int\limits_{F_t}  \langle R(\psi),\overline{\Theta}_N\rangle y^{-1} \,dx\,dy.$$

It follows from  Proposition \ref{RC2} that \begin{align*}&\lim_{t\to \infty} \int\limits_{F_t}  \langle R(\psi),\overline{\Theta}_N\rangle y^{-1} \,dx\,dy= \\
&\lim_{t\to\infty}\int\limits_{-1/2}^{1/2} \langle \psi(x+it),\overline{\Theta_N(x+it)}\rangle t^{-1} \,dx=0.\end{align*} This finishes the proof.

\end{proof}

\section{Embedding trick}
\begin{theorem}\label{Th trick}We let $N$ be an  even lattice of signature $(2,0)$ and let $f$ be a $\C[N'/N]$-valued weakly holomorphic modular form of weight 1 with zero constant term and rational Fourier coefficients. Then there exists an even lattice $P$ of signature $(2,1)$ and a function $h\in M_{1/2}^{\,!}(\slz,\rho_P)$ such that \begin{description}\item{1.} There is an inclusion $N\subset P$ \item{2.} Lattice  $P$ contains  primitive norm zero vector\item{3.} Function $h$ has rational Fourier coefficients \item{4.} The constant term of $h$ is zero \item{5.} We have $T_{P,N}(h)=f$ for the map $T_{P,N}$ defined in  Theorem \ref{Th see-saw}.
\end{description}
\end{theorem}
\begin{proof}
We adopt the method explained in \cite{Bo1}, Lemma 8.1.

Consider two even unimodular definite lattices of dimension 24, say  three copies of $E_8$ root lattice $ E_8\oplus E_8\oplus E_8$ and a Leech lattice $\Lambda_{24}$. We can embed both lattices into $\frac{1}{16}\Z^{24}$. To this end we use the standard representation of $E_8$ in which all vectors have half integral coordinates. Use standard representation of Leech lattice and norm doubling map defined in \cite{Conway} Chapter 8, p.242.

Denote by  $M_1$ and $M_2$ the  negative definite lattices obtained from  $E_8\oplus E_8\oplus E_8$ and  $\Lambda_{24}$ by multiplying norm with $-1$ and assume that they are embedded into $\frac{1}{16}\Z^{24}$. Denote by $M$ the negative definite lattice $16 \Z^{24}$. Theta functions of lattices $M_1$ and $M_2$ are modular forms of level 1 and weight 12 and their difference is $720\Delta$, where $\Delta=q-24q^2+252q^3+O(q^4)$ is a  unique cusp form of level 1 and weight 12.


Consider the function $g$ in $M^!_{-11}(\slz,\rho_{N\oplus M})$ defined as $$g:= \res_{(N\oplus M_1)/N\oplus M}(f/\Delta)-\res_{(N\oplus M_2)/N\oplus M}(f/\Delta).$$ The maps $$\res_{(N\oplus M_i)/N\oplus M}:M^!_{-11}(\slz,\rho_{N\oplus M_i})\to M^!_{-11}(\slz,\rho_{N\oplus M}),\;i=1,2,$$ are defined as in Lemma \ref{res/tr}. 
It is easy to see that
$$T_{N\oplus M,N}(g)=T_{N\oplus M,N}(\res_{(N\oplus M_1)/N\oplus M}(f/\Delta)-\res_{(N\oplus M_2)/N\oplus M}(f/\Delta))=$$ $$T_{(N\oplus M_1)/N\oplus M}(f/\Delta)-T_{(N\oplus M_2)/N\oplus M}(f/\Delta)=$$ $$ \frac{f}{\Delta}(\overline{\Theta}_{M_1}-\overline{\Theta}_{M_2})=720 f. $$
Suppose that $g$ has a Fourier expansion $$g_\mu(\tau)=\sum_{m\in\Q} c_\mu(m)\e(m\tau),\;\mu\in(N'\oplus M')/ (N\oplus M).$$
By the construction of $g$ its constant term is zero.
Consider the following set of vectors in $M'$
$$S:=\{l\in M'| c_{(0,l+M)}(l^2/2)\neq 0\}, $$
where $(0,l+M)$ denotes an element in $(N'\oplus M')/(N\oplus M)$. Note that this set is finite and does not contain the zero vector.
Choose a vector $p\in M$ such that \\
1. the lattice $N\oplus \Z p$ contains a primitive norm $0$ vector;\\
2. $(p,l)\neq 0$ for all $l\in S$.\\
Consider the lattice $P:=N\oplus \Z p$. It follows from Theorem B1 that the subvariety $G(P)$ of $G(N\oplus M)$  is not  contained in the singular locus of $\Phi_{N\oplus M}(v^+,g)$. Moreover, the restriction of $\Phi_{N\oplus M}(v^+,g)$
to $G(P)$ is nonsingular at the point $G(N)$.

Define $h:=\frac{1}{720} T_{N\oplus M,P}(g)$. The constant term of $h$ is nonzero and $h$ has rational(with denominator bounded by 720) Fourier coefficients. We have

$$ T_{P,N}(h)=T_{P,N}(T_{N\oplus M,P}(g))=\frac{1}{720}T_{N\oplus M,N}(g)=f.$$
This finishes the proof.\end{proof}



\begin{theorem}\label{Cor trick} 

We let $N$ be an  even lattice of signature $(2,0)$ and let $f$ be a $\C[N'/N]$-valued weakly holomorphic modular form of weight 1 with zero constant term and rational Fourier coefficients. Then, $$\Phi_N(f)=\log \alpha$$ for some $\alpha\in\bar{\Q}$.
 \end{theorem}
 \begin{proof}Let $P$ and $h$ a lattice and a modular form  constructed in  Theorem \ref{Th trick}. Since the constant term of $h$ is zero, from Theorem B3 we know that
     $$\Phi_P(G(N),h)=-4\log|\Psi_P(\tau_N,h)|,$$ where $\Psi_P(\tau,h)$ is a meromorphic modular function on $\h$ for a subgroup of $\slz$  with respect to some unitary character and $\tau_N\in\h$ is a CM point. Theorem 14.1 of \cite{Bo2} says that this unitary character is finite.
     Theorem B3 Part 3 implies that $\Psi_P(\tau,h)$ has algebraic Fourier coefficients. Thus, it follows from the theory of complex multiplication that $\alpha:=\Psi_P(\tau_N,h)$ is an algebraic number.\end{proof}

\section{Main Theorem}

\begin{theorem}Let $\mathfrak{z}_1,\mathfrak{z}_2\in\h$ be two CM points in the same quadratic imaginary field $\Q(\sqrt{-D})$ and let $\bm{\lambda}$ be a relation on $S_{2k}(\slz)$ for integer $k>1$. Then there is an algebraic number $\alpha$ such
that $$G_{k,\,\bm{\lambda}}(\mathfrak{z}_1,\mathfrak{z}_2)=\log\alpha.$$
\end{theorem}
\begin{proof}Let $g_{\bm{\lambda}}$ be the weakly holomorphic modular form of weight $2-2k$ defined by Proposition \ref{g_lambda}. Consider a function $h_{\bm{\lambda}}=R^{k-1}(g_{\bm{\lambda}})$. In Theorem \ref{Th g}  we show that  \begin{equation}\label{T1}G_{k,\,\bm{\lambda}}(\tau_1,\tau_2)= \Phi_{M_2(\Z)}(v^+(\tau_1,\tau_2),h_{\bm{\lambda}})\end{equation} for $(\tau_1,\tau_2)\in\h\times\h\setminus S_{\bm{\lambda}} $.

Let $v^+(\mathfrak{z}_1,\mathfrak{z}_2)$ be a two dimensional  positive definite subspace of $M_2(\R)$ defined in \eqref{v+(tau1,tau2)}.  In the case when $\mathfrak{z}_1$ and $\mathfrak{z}_2$ lie in the same quadratic imaginary field the subspace $v^+(\mathfrak{z}_1,\mathfrak{z}_2)$ defines a rational splitting of $M_2(\Z)\otimes\R$. So, the lattice $N:=v^+(\tau_1,\tau_2)\cap M_2(\Z)$ has signature $(2,0)$.

It follows from Theorem \ref{Th f} that \begin{equation}\label{T2}\Phi_{M_2(\Z)}(v^+(\mathfrak{z}_1,\mathfrak{z}_2),R^{k-1} (g_{\bm{\lambda}}))=\Phi_{N}(f),\end{equation} where $f=(f_\nu)_{\nu\in N'/N}\in M^{\,!}_1(\slz,\rho_N)$ is given by $$f_\nu=[\,g_{\bm{\lambda}},\vartheta_{\nu}]_{k-1}.$$

Let $P$ and $h$ be as in Theorem \ref{Th trick}. Theorem \ref{Cor trick} implies

\begin{equation*}\Phi_{N}(f)=\Phi_{P}(G(N),h)=-4\log(|\Psi_P(G(N),h)|)\end{equation*}
and
\begin{equation}\label{T3}\Phi_{N}(f)=\log(\alpha) \end{equation}
for $\alpha\in\bar{\Q}$. Statement of the theorem follows from equations \eqref{T1}\,--\,\eqref{T3}.

     \end{proof}


\end{document}